\documentclass[12pt]{amsart}
\usepackage{amscd,amsmath,amsthm,amssymb}
\usepackage{color}
\usepackage{amsfonts,amssymb,amscd,amsmath,enumerate,verbatim,enumitem}
\usepackage{lineno}
 \def\NZQ{\mathbb}               
 \def\NN{{\NZQ N}}
 
 \def\ZZ{{\NZQ Z}}

 \def\PP{{\NZQ P}}

 \def\frk{\mathfrak}               

 \def\mm{{\frk m}}

 \def\G{{\mathcal G}}

 \def\ab{{\mathbf a}}

 \def\opn#1#2{\def#1{\operatorname{#2}}} 
 \opn\chara{char} \opn\length{\ell} \opn\pd{pd} \opn\rk{rk}
 \opn\projdim{proj\,dim} \opn\injdim{inj\,dim} \opn\rank{rank}
 \opn\depth{depth} \opn\grade{grade} \opn\height{height}
 \opn\embdim{emb\,dim} \opn\codim{codim}
 
 \opn\Tr{Tr} \opn\bigrank{big\,rank}
 \opn\superheight{superheight}\opn\lcm{lcm}
 \opn\trdeg{tr\,deg}
 \opn\reg{reg} \opn\lreg{lreg} \opn\ini{in} \opn\lpd{lpd}
 \opn\size{size} \opn\sdepth{sdepth}
 \opn\link{link}\opn\fdepth{fdepth}\opn\lex{lex}
 \opn\tr{tr}
\opn\Ap{Ap}
 \opn\div{div} \opn\Div{Div} \opn\cl{cl} \opn\Cl{Cl}
 \opn\Spec{Spec} \opn\Supp{Supp} \opn\supp{supp} \opn\Sing{Sing}
 \opn\Ass{Ass} \opn\Min{Min}\opn\Mon{Mon}
 \opn\Ann{Ann} \opn\Rad{Rad} \opn\Soc{Soc}
 \opn\Im{Im} \opn\Ker{Ker} \opn\Coker{Coker} \opn\Am{Am}
 \opn\Hom{Hom} \opn\Tor{Tor} \opn\Ext{Ext} \opn\End{End}
 \opn\Aut{Aut} \opn\id{id}
 
 \opn\nat{nat}
 \opn\pff{pf}
 \opn\Pf{Pf} \opn\GL{GL} \opn\SL{SL} \opn\mod{mod} \opn\ord{ord}
 \opn\Gin{Gin} \opn\Hilb{Hilb}\opn\sort{sort}
 \opn\PF{PF}\opn\Ap{Ap}
 \opn\tail{tail}
 \opn\length{length}
 \opn\aff{aff} \opn
 \con{conv} \opn\relint{relint} \opn\st{st}
 \opn\lk{lk} \opn\cn{cn} \opn\core{core} \opn\vol{vol}  \opn\inp{inp} \opn\nilpot{nilpot}
 \opn\link{link} \opn\star{star}\opn\lex{lex}\opn\set{set}
 \opn\width{wd}
 \opn\Fr{F}
 \opn\QF{QF}
 \opn\G{G}
 \opn\type{type}\opn\res{res}
 \opn\gr{gr}
 
 \def\pot#1#2{#1[\kern-0.28ex[#2]\kern-0.28ex]}

 \opn\dirlim{\underrightarrow{\lim}}
 \opn\inivlim{\underleftarrow{\lim}}
 \let\to=\rightarrow
 
 \def\Implies{\ifmmode\Longrightarrow \else
         \unskip${}\Longrightarrow{}$\ignorespaces\fi}
 \def\implies{\ifmmode\Rightarrow \else
         \unskip${}\Rightarrow{}$\ignorespaces\fi}
 \def\iff{\ifmmode\Longleftrightarrow \else
         \unskip${}\Longleftrightarrow{}$\ignorespaces\fi}

 \let\:=\colon
 \newtheorem{Theorem}{Theorem}[section]
 \newtheorem{Lemma}[Theorem]{Lemma}
 \newtheorem{Corollary}[Theorem]{Corollary}
 \newtheorem{Proposition}[Theorem]{Proposition}
 \newtheorem{Remark}[Theorem]{Remark}

 \let\epsilon\varepsilon
 \let\kappa=\varkappa
 \def\qed{\ifhmode\textqed\fi
       \ifmmode\ifinner\quad\qedsymbol\else\dispqed\fi\fi}
 \def\textqed{\unskip\nobreak\penalty50
        \hskip2em\hbox{}\nobreak\hfil\qedsymbol
        \parfillskip=0pt \finalhyphendemerits=0}
 \def\dispqed{\rlap{\qquad\qedsymbol}}
 
 \opn\dis{dis}
 \def\pnt{{\raise0.5mm\hbox{\large\bf.}}}
 
 \opn\Lex{Lex}

\begin{document}
\title {Cohen-Macaulay criteria for projective monomial curves via Gr\"obner bases}

\author {J\"urgen Herzog,  Dumitru I.\ Stamate}

\address{J\"urgen Herzog, Fachbereich Mathematik, Universit\"at Duisburg-Essen, Campus Essen, 45117
Essen, Germany} \email{juergen.herzog@uni-essen.de}

\address{Dumitru I. Stamate, Faculty of Mathematics and Computer Science, University of Bucharest, Str. Academiei 14, Bucharest -- 010014, Romania }
\email{dumitru.stamate@fmi.unibuc.ro}

\dedicatory{ }

\begin{abstract}
We prove  new characterizations based on Gr\"obner bases for the Cohen-Macaulay property of a projective monomial curve.
\end{abstract}

\thanks{}




\maketitle

\setcounter{tocdepth}{1}
\tableofcontents

\section*{Introduction}

Let $K$ be any field.
For any sequence of  distinct positive integers $\ab: a_1, \dots, a_n$ we denote $I(\ab)$ the kernel of the $K$-algebra homomorphism $\phi: S \to K[t]$
where $S=K[x_1, \dots, x_n]$ and $\phi(x_i)=t^{a_i}$ for $i=1, \dots, n$. The image of this map is the semigroup ring over $K$ of the semigroup $H$ generated by $a_1, \dots, a_n$.
We do not insist that $\ab$ is a minimal generating set for $H$.

In the following, we assume that $\gcd(a_1, \dots, a_n)=1$ and $a_n> a_i$ for all $i<n$.
We note that the homogenization of $I(a)$ with respect to the variable $x_0$  is again a toric ideal, namely it is the kernel of the $K$-algebra map $\psi: S[x_0]\to K[s,t]$
where $\psi(x_i)=t^{a_i}s^{a_n-a_i}$ for $1\leq i \leq n$ and $\psi(x_0)=s^{a_n}$.
The image of the map $\psi$ is the subalgebra $K[\mathcal{A}]$ of $K[t,s]$ generated by the monomials whose exponents are read from the columns of the matrix
\begin{equation}
\label{eq:A}
\mathcal{A}= \left(
\begin{matrix} 0 &a_1 & \dots  & a_{n-1} & a_n \\
a_n& a_n-a_1 &\dots &a_n- a_{n-1} & 0
\end{matrix}
\right).
\end{equation}
In case $K$ is an algebraically closed field, $I(\ab)$ is the vanishing ideal of the affine curve $C(\ab)$ given
parametrically by $t\mapsto (t^{a_1}, \dots, t^{a_n})$, while $I(\ab)^h$ is the vanishing ideal of the projective closure $\overline{C(\ab)}$ of $C(\ab)$ in $\mathbb{P}^n$,
 given parametrically by $[s:t] \mapsto[s^{a_n}: t^{a_1}s^{a_n-a_1}:\dots :t^{a_{n-1}}s^{a_n-a_{n-1}}: t^{a_n}]$. Curves of this type are called projective monomial curves.

The projective monomial curve $\overline{C(\ab)}$ is called arithmetically Cohen-Macaulay if its vanishing ideal  $I(\ab)^h$ is a Cohen-Macaulay ideal.
It is known that this is the case if and only if $t^{a_n},s^{a_n}$ is a regular sequence on $K[\mathcal{A}]$. This fact is a special case of \cite[Theorem 2.6]{GSW}.

 Arithmetically Cohen-Macaulay curves are not rare among the projective monomial curves.  It follows from  Vu's \cite[Theorem 5.7]{Vu}  that for any fixed $\ab$,
the curve $\overline{C(a_1+k, \dots, a_n+k)}$ is arithmetically Cohen-Macaulay for all $k\gg 0$.
In small embedding dimension,  Bresinsky, Schenzel and Vogel \cite{BSV} characterized the arithmetically Cohen-Macaulay projective monomial curves in $\PP^3$  by the property that $I(\ab)^h$ is generated by at most 3 elements.

In the context of numerical semigroups, Gr\"obner bases   have been used in
algorithms in \cite{Morales} and \cite{Roune}  to find  the Frobenius number of the numerical semigroup $H$ generated by $\ab$, or to characterize when is the
tangent cone of the semigroup algebra $K[H]$ Cohen-Macaulay, see \cite{Arslan}, \cite{AMS}.

One of   the main results in this paper is that $\overline{C(\ab)}$ is arithmetically Cohen--Macaulay if and only if $\ini(I(\ab)^h)$, respectively $\ini(I(\ab))$,
is a Cohen-Macaulay ideal, see Theorem~\ref{thm:main} (b) and (c).
Here the initial ideals are taken for a  reverse lexicographic order  for which $x_n$, $x_0$, respectively $x_n$ are the smallest variables.
 These conditions are also equivalent to condition (g) which says that
$$\ini(x_n, I(\ab))=(x_n,\ini(I(\ab))).$$
 Yet other equivalent properties are (d) and (e), namely that $x_n$ , respectively $x_0, x_n$, do not divide  any minimal monomial generator of $\ini(I(\ab))$ and of $\ini(I(\ab)^h)$, respectively, where the monomial orders are as before.

A Cohen-Macaulay criterion for a  simplicial semigroup ring  in terms of the  Gr\"obner basis of its  defining ideal is given by Kamoi  in \cite[Corollary 2.9]{Kamoi-nagoya} and \cite[Theorem 1.2]{Kamoi-commalg}.
In the particular case  considered in this paper,  equivalences (a), (d) and  (e) in Theorem \ref{thm:main} sharpen Kamoi's mentioned results and his \cite[Corollary 3.6]{Kamoi-commalg}.


The  dual sequence of $\ab$ is defined to be the sequence $\ab': a_n-a_1, \dots, a_n-a_{n-1}, a_n$.
The projective monomial curves associated to the  sequences $\ab$ and $\ab'$ are obviously  isomorphic.
So it is natural  to compare the ideals $I(\ab)$ and $I(\ab')$ and their reduced Gr\"obner bases.
That is the focus of Section \ref{sec:dual}.

For $w=(w_1, \dots, w_n) \in \ZZ^n$ we denote $\langle w, \ab \rangle= \sum_{i=1}^n  w_i a_i$, and we set $L(\ab)=\{ w \in \ZZ^n:  \langle w, \ab \rangle=0 \}$.
Obviously, $I(\ab)$ is just the lattice ideal of the lattice $L(\ab)$. Indeed, $I(\ab)$ is generated  by the binomials $f_w=x^{w^+}-x^{w^-}$ with $w\in L(\ab)$.
Let $\sigma:\ZZ^n \to \ZZ^n$ be the map given by $\sigma(w_1,\dots, w_n)=(w_1, \dots, w_{n-1}, -\sum_{i=1}^n w_i)$.
Then  $\sigma$ is an automorphism of the group $\ZZ^n$ such that  $\sigma^2=\id_{\ZZ^n}$ which induces an isomorphism between $L(\ab)$ and $L(\ab')$.
In particular,  $L(\ab')=(f_{\sigma(w)}\:\; w\in L(\ab))$.

In general,  a minimal set of binomial generators of $L(\ab)$ is not mapped to a minimal set of binomial generators of $L(\ab')$, see Remark \ref{rem:smth}.
However, in Theorem~\ref{thm:duality} we show that $\overline{C(\ab)}$ is arithmetically Cohen-Macaulay if and only if $\ini(f_{\sigma(w)})=\ini(f_w)$ for all $f_w \in \mathcal{G}$, where  $\mathcal{G}$ denotes the reduced Gr\"obner basis of $I(\ab)$ with respect to a reverse lexicographic monomial order with $x_n$ the smallest variable.  Moreover,  these conditions are also equivalent to the fact that  $f_w \in \mathcal{G}$ if and only if $f_{\sigma(w)} \in \mathcal{G}'$, for all $w \in \ZZ^n$, where $\mathcal{G}'$  is the reduced Gr\"obner basis of $I(\ab')$ with respect to the same monomial order.

Let $H$ denote the numerical semigroup generated by $\ab$. For any nonzero element  $h$ in $H$ its Ap\'ery set is defined as  $\Ap(H,h)=\{x \in H: x-h \notin H\}$.
For $h\in \Ap(H,a_n)$ we denote $\varphi_\ab(h)$ the smallest monomial in $S$ for the reverse lexicographic order such that its $\ab$-degree equals $h$. The close relationship
between the ideals $I(\ab)$ and $I(\ab')$ is also outlined by the fact that the curve $\overline{C(\ab)}$
is arithmetically Cohen-Macaulay if and only if  $$\ini(x_n, I(\ab)) =\ini(x_n, I(\ab')),$$ see Theorem \ref{thm:apery-cm}. Here one uses a reverse lexicographic order with $x_n$ the smallest variable.
For the proof, a key observation is that the latter equation is equivalent to the fact that for all $h$ in $\Ap(H,a_n)$ the $\ab'$-degree of $\varphi_\ab(h)$ is in $\Ap(H', a_n)$, where $H'$
denotes the semigroup generated by the dual sequence $\ab'$.  As a consequence, in Corollary \ref{cor:cn} we recover a criterion of Cavaliere and Niesi (\cite{CN})  for $\overline{C(\ab)}$ to be arithmetically Cohen-Macaulay,  see also \cite[Example 1.4]{RGU} .

When $n=3$, it is known from \cite{He-semi} that $\mu(I(\ab)) \leq 3$.
However, we give examples showing that  a reduced Gr\"obner basis may be arbitrarily large, see Proposition \ref{prop:big-ini}.
In case $\overline{C(\ab)}$ is arithmetically Cohen-Macaulay, in Proposition \ref{initialbound} we show that $\mu(\ini(I(\ab))) \leq {a_n \choose n-2}$.

In Section \ref{sec:applications} we apply Theorem \ref{thm:main} to test the arithmetically Cohen-Macaulay property for two families of
projective monomial curves in $\PP^3$ that have appeared in the literature.
For these families of 4-generated numerical semigroups which were introduced by Arslan (\cite{Arslan}) and by Bresinsky (\cite{Bresinsky}), respectively, we show that
the corresponding projective monomial curve is  (respectively, is not) arithmetically Cohen-Macaulay.

\section{A Cohen-Macaulay criterion via Gr\"obner bases}

The following lemma appears in \cite[ Exercise 5, page 392]{CLO}. Lacking reference to a proof,  we prefer to  provide one.
\begin{Lemma}
\label{lemma:revlex}
Let $I$ be an ideal in the polynomial ring $S=K[x_1,\dots, x_n]$ and $I^h \subset S[x_0]$ its homogenization with respect to the variable $x_0$.
We denote $<$ any reverse lexicographic monomial order on $S$ and $<_0$ the reverse lexicographic  monomial order on $S[x_0]$  extended from $S$ and such that $x_0$ is the least variable.

If $f_1, \dots, f_r$ is  the reduced Gr\"obner basis for $I$ with respect to $<$, then
$f_1^h, \dots, f_r^h$ is the reduced Gr\"obner basis for  $I^h$ with respect to $<_0$.
Moreover, $\ini_{<_0}(I^h) = (\ini_< (I) )S[x_0]$.
\end{Lemma}

\begin{proof}
Let $\mathcal{F}^h= \{f_1^h, \dots, f_r^h\}$.
It is proved in \cite[Proposition 3.15]{EH} that $\mathcal{F}^h$ is  a Gr\"obner basis for $I^h$ with respect to the block order  $<'$ on $S[x_0]$ which is defined as
$$
x^\alpha x_0^a <' x^\beta x_0^b \text{ if } (x^\alpha <x^\beta) \text{ or } (x^\alpha=x^\beta \text{ and }a<b)
$$
for all $\alpha, \beta \in \NN ^n$ and all nonnegative integers $a,b$.

Let $f$ be a nonzero polynomial in $I$. We write $f=m_1+\dots+ m_q$ as a sum of monomials with $m_i>m_{i-1}$ for $2\leq i \leq q$.
Then $\deg m_i \geq \deg m_{\color{blue} i-1}$ for $2\leq i \leq q$ and
$$
f^h=m_1+m_2 x_0^{\varepsilon_2}+\dots+ m_q x_0^{\varepsilon_q},
$$ where $\varepsilon_i= \deg m_1-\deg m_i$ for $i=2, \dots, q$.
Moreover, in the above decomposition of $f^h$  the monomials are listed decreasingly with respect to $<_0$.
Thus $\ini_{<_0}( f^h)=m_1=\ini_< (f)=\ini_{<'} (f^h)$ for all $f$ in $I$. It follows that
\begin{eqnarray*}
\ini_{<'} (I^h)= \ini_{<'}(f_i^h: 1\leq i \leq r)  &=&(\ini_{<'}(f_i^h):1\leq i \leq r)  \\
&=& (\ini_{<_0} (f_i^h): 1\leq i \leq r) \subseteq \ini_{<_0}(I^h).
\end{eqnarray*}
 Since the homogeneous ideals $\ini_{<'} (I^h)$ and $\ini_{<_0}(I^h)$ have the same Hilbert function,
 we conclude that they are equal and that  $\mathcal{F}^h$ is a Gr\"obner basis for $I^h$ with respect to $<_0$.

For any nonzero $f$ in $S$, $\tail(f)$ denotes the difference between $f$ and its leading term. 

Assume there exist $i, j$ such that $\ini_{<_0} (f_i^h)$ divides  a monomial in $\tail(f_j^h)$, i.e. $\ini_{<}(f_i)$ divides $m x_0^\varepsilon$ with $m$ a monomial in $\tail(f_j)$.
This implies that $\ini_< (f_i)$ divides $m$, which contradicts the fact that $\mathcal{F}$ is the reduced Gr\"obner basis for $I$ with respect to $<$.
Therefore $\mathcal{F}^h$ is the  reduced Gr\"obner basis for $I^h$ with respect to $<_0$.
\end{proof}

The following theorem is one of the main results of this paper.
For any monomial ideal $I$ we let $G(I)$ denote the unique minimal set of monomial generators for $I$.

\begin{Theorem}
\label{thm:main} Let $\ab: a_1, \dots, a_n$ be a sequence of positive integers with $a_n>a_i$ for all $i<n$.
Denote $<$ any reverse lexicographic order on $S=K[x_1, \dots, x_n]$ such  that $x_n$ is the smallest variable and $<_0$ the induced reverse lexicographic order on $S[x_0]$, where $x_n>x_0$.
The following conditions are equivalent:
\begin{enumerate}
\item [{\em (a)}]the projective monomial curve $\overline{C(\ab)}$ is  arithmetically Cohen-Macaulay;
\item [{\em (b)}] $\ini_{<_0}( I(\ab)^h)$ is a Cohen-Macaulay ideal;
\item [{\em (c)}] $\ini_< (I(\ab))$ is a  Cohen-Macaulay ideal;
\item [{\em (d)}] $x_n$ does not divide any element of $G(\ini_< (I(\ab)))$;
\item [{\em (e)}]$x_n$ and $x_0$ do not divide any element of $G(\ini_{<_0}(I(\ab)^h))$;
 \item [{\em (f)}] $x_n$ does not divide any element of $G(\ini_{<_0}(I(\ab)^h))$;
\item [{\em (g)}]$\ini_<(x_n, I(\ab))= (x_n, \ini_<(I(\ab))$.
\end{enumerate}
\end{Theorem}

\begin{proof}
Lemma \ref{lemma:revlex} implies that $G(\ini_<(I(\ab)))= G(\ini_{<_0} (I(\ab)^h))$. Therefore (b) \iff (c) and (d) \iff    (f) \iff (e).
The implication  (b) \implies (a) is a general fact, see for example \cite[Corollary 3.3.5]{HH}.
Assuming (e), we get that $x_0, x_n$ is a regular sequence modulo $\ini_{<_0}(I(\ab)^h)$, which implies (b).

Since $x_n$ is regular on $S/I(\ab)$, which is a domain, using \cite[Proposition 1.4]{CS} we have that $x_n$ is regular on $S/\ini_<(I(\ab))$ if
and only if $\ini_<(x_n, I(\ab))= (x_n, \ini_<(I(\ab))$. This shows (d) \iff (g).

It remains to prove that (a) \implies (e). It is known that the ring $K[\mathcal{A}]$ is Cohen-Macaulay if and only if $s^{a_n},t^{a_n}$ is a regular sequence on it, see \cite[Lemma 2.4]{GSW}.
That is equivalent to  $x_0, x_n$ being a regular sequence on $S[x_0]/I(\ab)^h$.
By a result of Bayer and Stillman (see \cite[Theorem 15.13]{Eis}),  this  implies (e).

\end{proof}

\section{Dual sequences}
\label{sec:dual}

Given $\ab: a_1, \dots, a_n$ a sequence  of distinct nonnegative integers such that $a_n >a_i$ for all $i<n$,
the dual sequence is defined to be $\ab': a_n-a_1, \dots, a_n-a_{n-1}, a_n$. It is clear that this procedure is a duality: $(\ab')'=\ab$.

The projective monomial curves associated to the  sequences $\ab$ and $\ab'$ are isomorphic.
Indeed, the ideals $I(\ab)^h$ and $I(\ab')^h$  are the kernel of the maps on $S[x_0]$ sending $x_0,\dots, x_n$
to the monomials having as exponent vectors the columns of the matrix $\mathcal{A}$ in \eqref{eq:A}, and respectively the columns of the matrix
$$
\label{eq:A'}
\mathcal{A'}= \left(
\begin{matrix} 0 &a_n-a_1 &   \dots &a_n-a_{n-1} & a_n \\
a_n& a_1  &\dots & a_{n-1} & 0
\end{matrix}
\right).
$$
This implies that the polynomials in the ideal $I(\ab')^h$ are obtained from those in $I(\ab)^h$ by switching the variables $x_0$ and $x_n$.

 In this section we compare the Gr\"obner bases of the ideals $I(\ab)$ and $I(\ab')$ with respect to a reverse lexicographic order, in connection to the Cohen-Macaulay property of the associated projective monomial curve.

Let $\sigma:\ZZ^n \to \ZZ^n$ be the map given by $\sigma(w_1,\dots, w_n)=(w_1, \dots, w_{n-1}, -\sum_{i=1}^n w_i)$. It is easy to see that $\sigma$ is an automorphism of the group $\ZZ^n$ and that  $\sigma^2=\id_{\ZZ^n}$.
 For $w=(w_1, \dots, w_n) \in \ZZ^n$ we denote $\langle w, \ab \rangle= \sum_{i=1}^n  w_i a_i$. We set
$$
L(\ab)=\{ w \in \ZZ^n:  \langle w, \ab \rangle=0 \}.
$$

\begin{Lemma} With notation as above, the map $\sigma$ induces an isomorphism between the groups $L(\ab)$ and $L(\ab')$.
\end{Lemma}

\begin{proof}
If $w=(w_1, \dots, w_n) \in \ZZ^n$ with $\sum_{i=1}^n w_i a_i=0$ then
$$
\left(\sum_{i=1}^n(a_n-a_i)w_i\right)-\left(\sum_{i=1}^n w_i \right)a_n=0,
$$
 equivalently $\langle \sigma(w), \ab' \rangle=0$ and $\sigma(w)\in L(\ab')$. Similarly, if $w' \in L(\ab')$ then $\sigma(w') \in L(\ab'')=L(\ab)$.
\end{proof}

If the entries of the vector $\alpha=(\alpha_1, \dots, \alpha_n)$ are nonnegative integers we let $x^\alpha=x_1^{\alpha_1} \cdots x_n^{\alpha_n}$.
For $w=(w_1,\dots, w_n) \in\ZZ^n$,  let $w^+$ and $w^-$ be the unique vectors with nonnegative entries and disjoint supports such that $w=w^+-w^-$. We denote $f_w=x^{w^+}-x^{w^-}$. It is clear that $f_{-w}=-f_w$. Therefore, a difference of two monomials with disjoint supports can be identified with a vector $w\in \ZZ^n$.

It is known that $I(\ab)=(f_w: w\in L(\ab))$, hence $I(\ab')=(f_{\sigma(w)}:w \in L(\ab))$.
However, it is not always true that $\sigma$ maps a minimal generating set (or a Gr\"obner basis)
 for $I(\ab)$ into a minimal generating set (or a Gr\"obner basis) for $I(\ab')$,
see Proposition \ref{prop:big-ini} and Remark \ref{rem:smth}.

\begin{Theorem}
\label{thm:duality}
Let $\ab:a_1, \dots, a_n$ be a sequence of nonnegative integers with $a_n>a_i$ for all $i<n$.
Let $\mathcal{G}$ and $\mathcal{G}'$ be the reduced Gr\"obner bases of $I(\ab)$ and $I(\ab')$, respectively, with respect to a reverse lexicographic monomial order on $S=K[x_1, \dots, x_n]$   such that $x_n$ is the smallest variable.
The following  conditions are equivalent:
\begin{enumerate}
\item[{\em(a)}] the projective monomial curve $\overline{C(\ab)}$ is arithmetically Cohen-Macaulay;
\item[{\em(b)}] $\ini(f_{\sigma(w)})=\ini(f_w)$, for all $f_w \in \mathcal{G}$;
\item[{\em(c)}]  $f_w \in \mathcal{G}$ \iff $f_{\sigma(w)} \in \mathcal{G}'$, for all $w \in \ZZ^n$.
\end{enumerate}
\end{Theorem}

\begin{proof}
We first prove that conditions (a) and (b) are equivalent.

(a)\implies (b):  We assume that $I(\ab)^h$ is a Cohen-Macaulay ideal in $S[x_0]$. We
pick $f_w$ in $\mathcal{G}$, where
$w=(w_1, \dots, w_n)$.  We denote $w'=\sigma(w)=(w'_1, \dots, w'_n)$.
Since the leading coefficient $LC(f_w)=1$ we get   that $\ini (f_w)=x^{w^+}$,
hence $d=\sum_{i=1}^n w_i \geq 0$ and $w'_n=-\sum_{i=1}^n w_i= -d \leq 0$.

 By Theorem \ref{thm:main} we obtain that $x_n$ does not divide $\ini(f_w)$, hence $w_n \leq 0$. Consequently, $(w')^+=w^+$.
Also, $\sum_{i=1}^n w_i'=-w_n \geq 0$, and hence $\deg x^{(w')^+} \geq \deg x^{(w')^-}$.
We distinguish two cases.

Firstly, if $w_n <0$ then $\deg x^{(w')^+} > \deg x^{(w')^-}$, hence $\ini(f_{w'})=x^{(w')^+}=x^{w^+}=\ini (f_w)$.
Moreover, $x^{(w')^{-}}=x^{w^-}\cdot x_n^{d+w_n}$.

Secondly, in case $w_n=0$ we get that $\deg x^{(w')^+} = \deg x^{(w')^-}$.
Now, if $d=0$ then $w'=w$.
If $d>0$, for the chosen monomial order we obtain that
$x^{(w')^+} > x^{(w')^-}$, because $x^{(w')^-}=x^{w^-} \cdot x_n^d$.

Thus  $\ini(f_{w'})=x^{w^+}=\ini (f_w)$ in all cases, and  property (b) holds.

(b) \implies (a): If $I(\ab)^h$ is not a Cohen-Macaulay ideal, by Theorem \ref{thm:main}
there exists $f_w$ in $\mathcal{G}$ such that  $x_n$ divides $\ini (f_w)$. Let $w=(w_1,\dots, w_n)$ and $w'=\sigma(w)=(w'_1, \dots, w'_n)$.
Since  $LC(f_w)=1$ we get that $\ini(f_w)= x^{w^+}$, hence $w_n >0$ and $\sum_{i=1}^n w_i \geq 0$.

There exists $i_0\neq n$ such that $w_{i_0}>0$, otherwise,
 since $\langle w,\ab\rangle=0$ we get that $\sum_{i=1}^{n-1} a_i (-w_i)=w_n a_n \geq -(w_1+\dots+w_{n-1})a_n > -\sum_{i=1}^{n-1}a_i w_i$, which is a contradiction.

Since $\sum_{i=1}^n w_i'=-w_n <0$, we obtain that $\deg x^{(w')^-}>\deg x^{(w')^+}$ and $\ini (f_{w'})=x^{(w')^-}$.
As $i_0<n$, we have that $w'_{i_0}=w_{i_0}>0$, hence $x_{i_0}$ divides $x^{(w')^+}$. On the other hand, condition (b) implies that
 $\ini (f_{\sigma(w)})=\ini (f_w)$, and therefore $x_{i_0}$  divides $x^{(w')^-}$ as well, which gives a contradiction.
We conclude that the projective monomial curve $\overline{C(\ab)}$ is  Cohen-Macaulay.

Next we prove that (a),(b) \implies (c).
  From the proof above of the equivalence (a)\iff (b),  we see that under the assumption that (a) (hence also (b)) holds,  for all $f_w$ in $\mathcal{G}$ one has
\begin{equation}
\label{eq:blah}
\tail (f_{\sigma(w)})=\tail (f_w) \cdot x_n^a \text{  for some integer }a .
\end{equation}
From Theorem \ref{thm:main} we  have  that $\ini(I(\ab'))=\ini(I(\ab))$, therefore property (b) implies that
 $\mathcal{G}''=\{ f_{\sigma(w)}: f_w\in \mathcal{G} \}$ is a minimal Gr\"obner basis for $I(\ab')$. We show that it is reduced.

Let $f_w \in \mathcal{G}$ with $w=(w_1, \dots, w_n)$ and $\sigma(w)=(w_1', \dots, w_n')$. Then $\ini(f_w)=x^{w^+}$ and $\deg (x^{w^+}) \geq \deg (x^{w^-})$.
Condition (a) and Theorem \ref{thm:main} imply that  $w_n \leq 0$. Thus $\sum_{i=1}^n w_i'= - w_n \geq 0$,
 which implies that $\deg(x^{(w')^+}) \geq \deg(x^{(w')^-})$.

If $w_n=0$ then $w'$ is homogeneous and either $w'=w$ (if $\sum_{i=1}^n w_i =0$), or $x_n$ divides $x^{(w')^-}$ (if $\sum_{i=1}^n w_i >0$), hence
$x^{(w')^+}=\ini(f_{\sigma(w)})$ and $LC(f_{\sigma(w)})=1$.

If $w_n<0$ then $\deg(x^{(w')^+}) >\deg( x^{(w')^-})$, and  again  $LC(f_{\sigma(w)})=1$.

If there are $f_w$ and $f_{\widetilde{w}}$ in $\mathcal{G}$ such that $\ini(f_{\sigma(w)})$ divides
$\tail(f_{\sigma(\widetilde{w})})$, then, since $x_n$ is not in the support of $\ini (f_w)=\ini (f_{\sigma(w)})$, by using \eqref{eq:blah} we get
that $\ini(f_w)$ divides $\tail(f_{\widetilde{w}})$. This contradicts the fact that $\mathcal{G}$ is the reduced Gr\"obner basis for $I(\ab)$.
Hence $\mathcal{G}''=\mathcal{G}'$, which proves (c).

For (c) \implies (a): If we assume that (c) holds, but $I(\ab)^h$ is not Cohen-Macaulay,
 then by Theorem \ref{thm:main} there exists $f_w$ in $\mathcal{G}$ such that
 $x_n$ divides $\ini(f_w)=x^{w^+}$.  Let $\sigma(w)=(w_1', \dots, w_n')$.  Since $\sum_{i=1}^n w_i'= - w_n <0$, it follows that $\deg(x^{\sigma(w)^-}) > \deg(x^{\sigma(w)^+})$, hence
$\ini(f_{\sigma(w) }) =x^{\sigma(w)^-}$. On the other hand, property (c) implies that $f_{\sigma(w)} \in \mathcal{G}'$, hence $\ini(f_{\sigma(w)})= x^{\sigma(w)^+}$, which is a contradiction.
Therefore, property (a) holds.

This ends the proof of the theorem.
\end{proof}

\begin{Remark}
{\em
The fact that the involution $\sigma$ maps some (minimal) Gr\"obner basis for $I(\ab)$ into a (minimal) Gr\"obner basis for $I(\ab')$ is not enough to  imply
that the curve  $\overline{C(\ab)}$  is arithmetically Cohen-Macaulay.

 Indeed, let $\ab: 4,13,19 $.
Then $\ab': 15, 6,19$. A Singular \cite{Sing} computation shows that
\begin{eqnarray*}
\mathcal{G} &=& \{y^5-x^2z^3, x^3y^2-z^2, x^5z-y^3, x^8-yz \}, \text{ and } \\
\mathcal{G'} &=& \{y^5-x^2, x^3y^2-z^3, y^3z^3-x^5,  x^8-yz^6\}
\end{eqnarray*}
are the reduced Gr\"obner bases with respect to the reverse lexicographic order with $x>y>z $ for the ideals $I(\ab)$ and $I(\ab')$, respectively. Therefore,  $\overline{C(\ab)}$ and  $\overline{C(\ab')}$ are not  arithmetically Cohen-Macaulay, by Theorem \ref{thm:main}.

One has  $\mathcal{G}=\{ f_{w_1}, f_{w_2}, f_{w_3},  f_{w_4} \}$, where   $w_1=(-2, 5,-3)$, $w_2=(3,2,-2)$, $w_3=(5,-3,1)$ and $w_4=(8,-1,-1)$.
Since $\sigma(w_1)= (-2,5,0)$, $\sigma(w_2)=(3,2,-3)$, $\sigma(w_3)=(5,-3,-3)$ and $\sigma(w_4)= (8,-1,-6)$
we note that
$$
\mathcal{G}'=  \{ f_{\sigma(w_1)}, f_{\sigma(w_2)}, -f_{\sigma(w_3)} , f_{\sigma(w_4)}\}.
$$
This means that $\{  f_{\sigma(w_1)}, f_{\sigma(w_2)}, f_{\sigma(w_3)} , f_{\sigma(w_4)}\}$ is a minimal Gr\"obner basis for $I(\ab')$, although different from the reduced one $\mathcal{G}'$.
}
\end{Remark}

Let $\ab=a_1, \dots, a_n$ be a sequence of nonnegative integers with $\gcd(a_1, \dots, a_n)=1$ and $a_n>a_i$ for all $i=1, \dots, n-1$.
Our next goal is to describe the Cohen-Macaulay property of $\overline{C(\ab)}$  in terms of the Ap\'ery sets of the semigroup generated by $\ab$ or the dual sequence $\ab'$.
We recall that for a numerical semigroup $H$ and $0\neq h$ in $H$, the Ap\'ery set of $H$ with respect to $h$ is
$$
\Ap(H,h)=\{ x\in H: x-h \notin H\}.
$$
It is known that  $|\Ap(H,h)|=h$ and that the elements of $\Ap(H,h)$ have distinct residues modulo $h$.

The sequence $\ab$ induces a grading on $S=k[x_1,\dots, x_n]$ by letting $\deg (x_i)=a_i$ for all $i=1, \dots, n$.
For any monomial $x^\alpha=x_1^{r_1}\cdots x_n^{r_n}$ we denote its $\ab$-degree
by $\deg_\ab(x^\alpha)=\sum_{i=1}^n r_i a_i$.
We denote $H$ the numerical semigroup generated by $\ab$.
For any $h$ in $\Ap(H, a_n)$ we denote $\varphi_\ab(h)$ the smallest monomial in $S$  (with respect to a reverse lexicographic monomial order
where $x_n$ is the smallest variable) such that its  $\ab$-degree equals $h$.
Since $h-a_n \notin H$ we see that the monomial $\varphi_\ab (h)$ is in $S'=k[x_1, \dots, x_{n-1}]$.

\begin{Proposition}
\label{prop:apery}
With notation as above, for any $h$ in $\Ap(H, a_n)$  the monomial $\varphi_\ab(h) \notin \ini (x_n,I(\ab))$.
\end{Proposition}

\begin{proof}
Let $h\in \Ap(H, a_n)$ and $\varphi_\ab(h)= x^\alpha$.  Assume that $x^\alpha \in \ini(x_n, I(\ab))$. Then $x^\alpha = \ini(F)$ for some $F$ in $(x_n, I(\ab))$.
Since the ideal $(x_n, I(\ab))$ is generated by  monomials and binomials which are the difference of two monomials with the same $\ab$-degree, without loss of generality we may assume that $F$ is $\ab$-homogeneous. Thus we may write
$$
x^\alpha= \ini(x_nf+f_1 f_{w_1}+ \dots +f_q f_{w_q}),
$$
where  $w_1, \dots, w_q \in L(\ab)$,  and $f, f_1, \dots, f_q$ are $\ab$-homogeneous with $h=\deg_\ab (x_n f)=\deg_\ab (f_1 f_{w_1})=\dots = \deg_\ab (f_q f_{w_q})$.

We notice that $f=0$, otherwise $h=\deg_\ab (x^\alpha)=\deg_\ab(x_n f)=  a_n+  \deg_\ab (f)$,  gives that $h-a_n \in H$,   which is false.
Hence  $x^\alpha \in \ini(I(\ab))$. The ideal $I(\ab)$ has a Gr\"obner basis of binomials, hence we can write $x^\alpha=m\cdot \ini(f_w)$ for some binomial $f_w \in I(\ab)$ and $m$ a monomial in $S'$.
Without loss of generality, we may assume $\ini (f_w)=x^{w^{+}}$. Thus $x^{w^+}>x^{w^{-}}$, which gives $\varphi_\ab(h)=x^\alpha > x^{w^{-}}m$. But $\deg_\ab (w^+)=\deg_\ab(w^-)$,
hence $\deg_\ab(x^\alpha)=\deg_\ab(x^{w^{-}}m)$, which contradicts the choice of $x^\alpha$.
Therefore,  $\varphi_\ab(h) \notin \ini (x_n,I(\ab))$.
\end{proof}

If we identify a monomial which is in $S$ and not in $\ini(x_n, I(\ab))$ with its residue class  modulo the monomial ideal $\ini(x_n, I(\ab))$, by Proposition \ref{prop:apery} the assignment $\varphi_\ab (-)$ defines a map
  from $ \Ap(H, a_n)$  into $\Mon(S/\ini(x_n, I(\ab)))$,   the $K$-basis of monomials of  $S/\ini(x_n, I(\ab))$. We prove that this is a bijection.

\begin{Proposition}
\label{prop:bijection}
The map $\varphi_\ab:  \Ap(H, a_n) \to \Mon(S/\ini(x_n, I(\ab)))$  is bijective.
\end{Proposition}

\begin{proof}
Let $h, h'$ in $\Ap(H, a_n)$ with $\varphi_\ab(h)=\varphi_\ab(h')$. Then $h=\deg_\ab (\varphi_\ab(h))= \deg_\ab (\varphi_\ab(h'))= h'$, and the map $\varphi_\ab$ is injective.

By Macaulay's theorem (\cite[Theorem 2.6]{EH}) the monomials in $S$ which do not belong to  $\ini(x_n, I(\ab))$ form a $K$-basis for $S/(x_n, I(\ab))$. Therefore,
\begin{eqnarray*}
  |\Mon(S/\ini(x_n, I(\ab)))|  &=& \dim_K S/(x_n, I(\ab)) \\
                     &=&  \dim_K K[H]/(t^{a_n})= a_n.
\end{eqnarray*}
Since   $|\Ap(H,a_n)|=a_n$, we conclude that the map $\varphi_\ab$ is bijective.
\end{proof}

\begin{Theorem}
\label{thm:apery-cm}
Let $\ab: a_1,\dots, a_n$ be a sequence of distinct positive integers such that $\gcd(a_1, \dots , a_n)=1$ and $a_n>a_i$ for all $i=1,\dots,n-1$. We denote $\ab'$  the dual sequence of $\ab$.
Let $H$ and $H'$ be the numerical semigroups generated by $\ab$ and $\ab'$, respectively.
The following statements are equivalent:
\begin{enumerate}
\item[{\em (a)}] the projective monomial curve $\overline{C(\ab)}$ is arithmetically Cohen-Macaulay;
\item[{\em (b)}] $\ini(x_n, I(\ab))= \ini(x_n, I(\ab'))$;
\item[{\em (c)}] $\Mon (S/\ini(x_n, I(\ab)) )= \Mon (S/\ini(x_n, I(\ab')) )$;
\item[{\em (d)}] $\deg_{\ab'} (\varphi_\ab(h)) \in \Ap(H', a_n)$ for all $h$ in $\Ap(H, a_n)$,
\end{enumerate}
where the initial ideals are taken with respect to the reverse lexicographic term order on $S$.
\end{Theorem}

\begin{proof}
Assume (a) holds. It follows from   Theorem \ref{thm:main} that  $\ini (x_n,I(\ab))=(x_n, \ini (I(\ab)))$ and $\ini (x_n,I(\ab'))=(x_n, \ini (I(\ab')))$.
We get from Lemma \ref{lemma:revlex} that $G(\ini (I(\ab)))= G(\ini (I^h(\ab)))= G(\ini (I^h(\ab'))) = G(\ini (I(\ab)))$, hence the statement (b) is true.

Clearly, the properties (b) and (c) are equivalent.

We now prove that (b) \iff (d). Assume that (b) holds. Let $h\in \Ap(H, a_n)$. By Proposition \ref{prop:apery} we have that the monomial $\varphi_\ab(h)$ is not in $\ini(x_n, I(\ab))$,
hence it is not in $\ini(x_n, I(\ab'))$. Using Proposition \ref{prop:bijection} we get that $\varphi_\ab(h)= \varphi_{\ab'}(h')$ for some $h'$ in $\Ap(H', a_n)$.
Hence $\deg_{\ab'}(\varphi_\ab(h)) \in \Ap(H', a_n)$, which proves (d).

Conversely, we assume that (d) holds and we consider the monomial $x^\alpha$ not in $\ini(x_n, I(\ab))$. By  Proposition \ref{prop:bijection} there exists $h$ in $\Ap(H, a_n)$ such that  $\varphi_\ab(h)=x^\alpha$.
Property (d) implies that there exists $h'$ in $\Ap(H', a_n)$ such that $h'=\deg_{\ab'}(x^\alpha)$, which by Proposition \ref{prop:apery} gives that $x^\alpha \notin \ini(x_n, I(\ab'))$.
Hence (d) \implies (b).

To finish the proof of the theorem we are left to show that (b) \implies (a). Assume  $\ini(x_n, I(\ab))= \ini(x_n, I(\ab'))$. By Theorem \ref{thm:main}, it is enough to prove
that $x_n$ does not divide any monomial in $G(\ini(I(\ab)))$.  Assume there exists a monomial $u \cdot x_n^c$ in $G(\ini(I(\ab)))$ with $u$ not divisible by $x_n$ and $c>0$.
Then $u$ is not a constant, otherwise, since $I(\ab)$ has a Gr\"obner basis of binomials, there exists $f=x_n^c- x_1^{r_1}\dots x_{n-1}^{r_{n-1}}$ in $I(\ab)$ with $x_n^c=\ini (f)$.
This implies that we have a relation $c \cdot a_n= \sum_{i=1}^{n-1} r_i a_i$ with $c \geq \sum_{i=1}^{n-1} r_i$, which is false since $a_n >a_i$ for all $i<n$.

Let $u\cdot x_n^c=\ini (f_w)$, where $w=(w_1, \dots, w_n) \in L(\ab)$. Without loss of generality we may assume $\ini(f_w)= x^{w^+}$, hence $w_n=c$.
 Set $v=x^{w^-}$ and $d=\deg (u\cdot x_n^c)-\deg v$. Then $d>0$ by the above discussion.
The sum of the components of $\sigma(w)$ equals $\sum_{i=1}^{n-1} w_i + (- \sum_{i=1}^n w_i)=-w_n=-c <0$, hence
$\deg (x^{\sigma(w)^+}) < \deg (x^{\sigma(w)^-})$ and $f_{\sigma(w)}=u-x_n^d\cdot v$.
 This gives that $$u= x_n^d\cdot v+ f_{\sigma(w)} \in (x_n, I(\ab')),$$
and also that $u\in \ini (x_n, I(\ab'))$, which  by our  hypothesis (b) implies that $u\in \ini (x_n, I(\ab))$. Since the ideal $I(\ab)$ is  $\ab$-homogeneous we can write
$$
u= \ini(x_nf+f_1 f_{z_1}+ \dots +f_q f_{z_q}),
$$
where  $z_1, \dots, z_q \in L(\ab)$,  and $f, f_1, \dots, f_q$ are $\ab$-homogeneous with $\deg_\ab (x_n f)=\deg_\ab (f_1 f_{z_1})=\dots = \deg_\ab (f_q f_{z_q})$.
We see that $f\neq 0$, otherwise $u\in \ini(I(\ab))$, which contradicts the fact that $u\cdot x_n^c$ is a minimal monomial generator for $\ini (I(\ab))$.

Let $h=\deg_\ab(u)$. Since $f\neq 0$ we get that $h-a_n \in H$.  We may write $h=h_1+\lambda_n a_n$ with  $\lambda_n$ a maximal positive integer and $h_1\in H$, i.e.  $h_1 \in \Ap(H, a_n)$.
Let $u_1=\varphi_\ab(h_1)= x_1^{\lambda_1}\dots x_{n-1}^{\lambda_{n-1}}$. Then the binomial $f_1=u_1x_n^{\lambda_n}-u$ is in $I(\ab)$.
As $u\cdot x_n^c \in G(\ini(I(\ab)))$ with $c>1$, we get that
$u \notin \ini(I(\ab))$, hence $\ini(f_1)=u_1\cdot x_n^{\lambda_n} \in \ini(I(\ab))$.

By Proposition \ref{prop:apery}, $u_1\notin \ini(x_n, I(\ab))$, hence $u_1 \notin \ini(I(\ab))$, as well.
This implies that $u_1\cdot x_n^{\lambda_n}$ is divisible by a monomial $u_2  x_n^e \in G(\ini(I(\ab))$ with $x_n$ and $u_2$ coprime, $e>0$. Therefore $u_2$ divides $u_1$,
hence $\deg_\ab(u_2)+h_2=\deg_\ab(u_1)$ for some positive $h_2$ in $H$. This gives $\deg_\ab(u_2) \in \Ap(H,a_n)$.

We may write $u_2\cdot x_n^e=\ini(f_{\widetilde{w}})$ with $\widetilde{w} \in L(\ab)$, and arguing as before we get that
$u_2= x_n^{d'}\cdot v_1+ f_{\sigma({\widetilde{w}})}\in \ini(x_n, I(\ab))$ for some positive $d'$. Thus
$$
u_2= \ini(x_nf'+f_1' f_{z_1'}+ \dots +f_\ell' f_{z_\ell'}),
$$
where  $z_1', \dots, z_\ell' \in L(\ab)$,  and $f', f_1', \dots, f_q'$ are $\ab$-homogeneous with $\deg_\ab (x_n f')=\deg_\ab (f_1' f_{z_1'})=\dots = \deg_\ab (f_\ell' f_{z_\ell'})$.
If $f'=0$, then $u_2\in \ini(I(\ab))$, which is false since $u_2\cdot x_n^e \in G(\ini(I(\ab)))$. On the other hand, $f'\neq 0$ implies that $a_n+ \deg_\ab(f')=\deg_\ab(u_2)$, hence
$\deg_\ab(u_2) \notin \Ap(H, a_n)$, which is also false.
Therefore $x_n$ does not divide any monomial in $G(\ini(I(\ab)))$. This concludes the proof of the implication (b) \implies (a) and of the theorem.
\end{proof}


Let $\Ap(H, a_n)=\{0, \nu_1, \ldots, \nu_{a_n-1}\}$. We may assume that $\nu_i\equiv i \mod a_n$ for all $i$.  For each $\nu_i$,  let $\mu_i \in H'$ be the smallest element such that $(\nu_i, \mu_i)$
 is in the semigroup generated by the columns of the matrix $\mathcal{A}$ from \eqref{eq:A}.

 Note that  $\mu_i\equiv -i\mod a_n$ for all $i$. Cavalieri and Niesi \cite{CN} call $\Ap(H, a_n)$ {\em good}, if  $\{0,\mu_1, \ldots, \mu_{a_{n}-1}\}=\Ap(H', a_n)$.

\medskip
As a consequence of Theorem~\ref{thm:apery-cm} we obtain

\begin{Corollary}(Cavaliere-Niesi, \cite[Theorem 4.6]{CN})
\label{cor:cn}
The  projective monomial curve $\overline{C(\ab)}$ is arithmetically Cohen-Macaulay if and only if $\Ap(H, a_n)$ is good.
\end{Corollary}

\begin{proof}
Let $\nu_i=\sum_{j=1}^{n-1}r_ja_j$ with integer coefficients $r_j\geq 0$ and $\sum_{j=1}^{n-1}r_j(a_n-a_j)$  minimal. Then $\mu_i=\sum_{j=1}^{n-1}r_j(a_n-a_j)$. Thus $\mu_i=(\sum_{j=1}^{n-1}r_j)a_n-\nu_i$ with  $\sum_{j=1}^{n-1}r_j$ minimal and $\sum_{j=1}^{n-1}r_ja_j=\nu_i$.

This shows that if $\varphi_{\ab}(\nu_i)=\prod_{j=1}^{n-1}x_j^{s_j}$, then
$\deg_{\ab'} (\varphi_\ab(\nu_i))= \sum_{i=1}^{n-1}s_j(a_n-a_j)=(\sum_{j=1}^{n-1} s_j) a_n-\sum_{j=1}^{n-1} s_j a_ j=  (\sum_{j=1}^{n-1} r_j)a_n-\nu_i=\mu_i$.
Hence Theorem~\ref{thm:apery-cm}(a)\iff (d) yields the desired conclusion.
\end{proof}

\section{A bound  for the number of generators of $\ini (I(\ab))$, when $\overline{C(\ab)}$ is arithmetically Cohen-Macaulay}
\label{sec:bound}

In this section we show by examples that the number of generators of $\ini (I(\ab))$ may be arbitrarily large, already  if $\ab$ has only $3$ elements.

\begin{Proposition}
\label{prop:big-ini}
For  the  integer $h\geq 2$, let $\ab= 4, 6h+1, 6h+7 $. Then $\mu(\ini(I(\ab)))= h+2$, where the initial ideal is computed with respect to the reverse lexicographic monomial order with $x_1>x_2>x_3$.
\end{Proposition}

\begin{proof}
We first find $I(\ab)$ using the method from \cite{He-semi}. For $1\leq i \leq 3$ we let $c_i$ be the smallest prositive integer such that
\begin{equation}
\label{eq:3semi}
c_i a_i= r_{ij}a_j+r_{ik}a_k,
\end{equation}
with $r_{ij}, r_{ik}$ nonnegative integers  and $\{i,j,k\}= \{1,2,3\}$.
Since $a_1, a_2, a_3$ are pairwise coprime, it is known from \cite{He-semi} that the $r_{ij}$'s are unique and positive, and  $c_i=r_{ji}+r_{ki}$ for all $\{i,j,k\}= \{1,2,3\}$.

From the equations $(3h+2) a_1=a_2+a_3$ and $2a_3=3a_1+ 2 a_2$ we find $c_1=3h+2$, $c_3=2$ and the corresponding $r_{ij}$'s from \eqref{eq:3semi}.
Hence $c_2=3$ and $3a_2=(3h-1)a_1+a_3$ is the corresponding equation from \eqref{eq:3semi}.
According to \cite{He-semi}, the ideal $I(\ab)$ is minimally generated by $f_1= x_1^{3h+2}-x_2 x_3$, $f_3=x_1^3x_2^2-x_3^2$ and $g_1=x_1^{3h-1}x_3-x_2^3$.

We introduce recursively the polynomials $g_{i+1}=S(g_i, f_3)$ for $1\leq i \leq h-1$. It follows easily by induction  that
$g_i=x_1^{3(h-i)+2}x_3^{2i-1}-x_2^{2i+1}$, for $1\leq i \leq h$.
We claim that
\begin{equation}
\label{eq:g}
\mathcal{G}=\{f_1, g_1, \dots, g_h, f_3\}
\end{equation}
is the reduced Gr\"obner basis of $I(\ab)$.

To prove that $\mathcal{G}$ is a Gr\"obner basis  we need to check that the $S$-pairs of elements in $\mathcal{G}$ reduce to zero with respect to $\mathcal{G}$, see \cite[Theorem 2.14]{EH}.
Here are the relevant computations.

$S(f_1, f_3)=x_2^2-x_1^{3h-1}f_3=x_3(x_2^3-x_1^{3h-1}x_3)= -x_3 g_1  \stackrel{\mathcal{G}}{\rightarrow} 0$.

 $S(g_h, f_1) \stackrel{\mathcal{G}}{\rightarrow} 0  \text{ since } \gcd(\ini(g_h)), \ini(f_1))=1$.

$S(g_h,f_3)=x_1^3g_h-x_2^{2h-1}f_3=x_3^2 g_{h-1} \stackrel{\mathcal{G}}{\rightarrow} 0$. \\
For  $1\leq i \leq h-1:$ \\
\indent $S(g_i, f_1) =x_1^{3i}g_i-x_3^{2i-1}f_1=x_2(x_3^{2i}-x_1^{3i}x_2^{2i})=x_2\cdot (x_3^2-x_1^3x_2^2)\cdot (\dots) =
 x_2\cdot f_3 \cdot (\dots)  \stackrel{\mathcal{G}}{\rightarrow} 0,  \text{ and}$

$S(g_i, f_3)=g_{i+1} \stackrel{\mathcal{G}}{\rightarrow} 0.$ \\
For $1\leq i< j<h$: \\
\indent $S(g_i, g_j) = x_3^{2(j-i)}g_i-x_1^{3(j-i)} g_j  = x_1^{3(j-i)}x_2^{2j+1}-x_3^{2(j-i)}x_2^{2i+1}
= x_2^{2(i+1)}(x_1^{3(j-i)}x_2^{2(j-i)}-x_3^{2(j-i)}) =x_2^{2(i+1)} \cdot f_3 \cdot (\cdots) \stackrel{\mathcal{G}}{\rightarrow} 0.$

 For $1\leq i <h$ we have that $S(g_i, g_h) \stackrel{\mathcal{G}}{\rightarrow} 0$,  since $\gcd(\ini(g_i), \ini(g_h))=1$.

By inspecting the binomials in $\mathcal{G}$ it follows that they are in fact the  reduced Gr\"obner basis for $I(\ab)$.
This shows that $\mu(\ini(I(\ab))= |\mathcal{G}|=h+2$.
\end{proof}

\begin{Remark}
\label{rem:smth}
{\em
The  sequence dual to the one from Proposition \ref{prop:big-ini} is $$\ab'=6h+3, 6, 6h+7.$$  It is easy to check (using \cite{He-semi}) that the corresponding toric ideal
is a complete intersection  $I(\ab')=(x_2^{2h+1}-x_1^2, x_1^3x_2^2-x_3^3)$.
In particular, this shows that  the image through the involution $\sigma$ of a  minimal set of binomial generators for $I(\ab)$  may no longer be a minimal generating system for  $I(\ab')$.

Arguing as in the proof of Proposition~\ref{prop:big-ini},  it is routine to verify that for the reduced Gr\"obner basis $\mathcal{G}$ in \eqref{eq:g}, the set
$\{ f_{\sigma(w)}: f_w \in \mathcal{G} \}$ is a minimal Gr\"obner basis for $I(\ab')$. The latter  set of binomials is fully interreduced, yet it is not the reduced Gr\"obner
basis for $I(\ab')$ since the leading coefficients of the binomials coming from $g_1, \dots, g_{h-1}$ equal $-1$.

From Theorem \ref{thm:duality} we infer that  $\overline{C(\ab)}$ is not arithmetically Cohen--Macaulay.
This can also be seen from the fact that  $\ini(g_1)=x_1^{3h-1}x_3 \in G(\ini(I(\ab)))$  (as $h>1$) and using Theorem \ref{thm:main}.
}
\end{Remark}

\medskip

If $\overline{C(\ab)}$ is arithmetically Cohen-Macaulay, we give an explicit  bound for
$\mu(\ini_<(I(\ab)))$ depending on $a_n$, the largest element of the sequence $\ab$. To prove this we first show

\begin{Lemma}
\label{bound}
Let $n\geq 2$ and  $I\subset S=K[x_1,\ldots,x_n]$ be a graded ideal  with $\mm^k\subset I$, where $\mm=(x_1,\ldots,x_n)$.
Let $<$ be a monomial order on $S$. Then $\mu(I)\leq \mu(\ini_<(I))\leq \mu(\mm^k)={n+k-1\choose n-1}$, and $\mu(I)={n+k-1\choose n-1}$ if and only if $I=\mm^k$.
\end{Lemma}

\begin{proof}
It suffices to show
that for a monomial ideal $J\subset S$ with $\mm^k\subset J$, one has $\mu(J)\leq \mu(\mm^k)$, and $\mu(J)=\mu(\mm^k)$, if and only if $J=\mm^k$.
 We prove this  by induction on $k-a$, where  $a$ is the least degree of a monomial generator of $J$.
If $a=k$, the assertion is trivial. Suppose now that $a<k$.
We denote by $G(J)$ the unique minimal set of monomial  generators of $J$, and set  $G(J)_j=\{u\in G(J)\:\; \deg u=j\}$ for all $j$,
and  let $J'=\mm J_a+J_{\geq a+1}$, where $J_{\geq a+1}=(u\in J\:\; \deg u\geq a+1)$.
Then $G(J')_j=0$ for $j<a+1$ and $\mm^k\subset J'$.
Therefore, by our induction hypothesis, we have $\mu(J')\leq \mu(\mm^k)$. On the other hand, $G(J')_{a+1}$ is
the disjoint union of $G(\mm J)_{a+1}$ and $G(J)_{a+1}$. Furthermore, $G(J')_j=G(J)_j$ for $j>a+1$.
Hence,  since  $|G(J)_a|<|G(\mm J)_{a+1}|$, it follows that $\mu(J)=|G(J)|<|G(J')|=\mu(J')\leq \mu(\mm^k)$, as desired.
\end{proof}

\begin{Proposition}
\label{initialbound}
Suppose that the monomial curve $\overline{C(\ab)}$ is arithmetically Cohen--Macaulay. Then $\mu(\ini(I(\ab))\leq  {a_n\choose n-2}$.
\end{Proposition}

\begin{proof}
As before we assume that $\ab=a_1,\ldots,a_n$ with $a_n>a_i$ for all $i$, and we let  $H$ be the numerical semigroup generated by $\ab$.
Then $I(\ab)\subset S=K[x_1,\ldots,x_n]$ and  $G(\ini_<(I(\ab))\subset \bar{S}=K[x_1,\ldots,x_{n-1}]$, by Theorem~\ref{thm:main}. Therefore,
\begin{eqnarray*}
\length(\bar{S}/\ini_<(I(\ab))&=& \length(S/(x_n,\ini_<(I(\ab))))=\length(S/(x_n,I(\ab)))=\\
\length(K[H]/(t^{a_n}))&=&a_n.\nonumber
\end{eqnarray*}

Let $k$ be the smallest number such that $(x_1,\ldots,x_{n-1})^k\subset \ini_<(I(\ab))$. Then
\begin{eqnarray*}
a_n&=& \sum_{j=0}^{k-1}\dim_K (\bar{S}/\ini_<(I(\ab))_j=  1+(n-1) +\sum_{j=2}^{k-1}\dim_K (\bar{S}/\ini_<(I(\ab))_j\\
&\geq &1+(n-1)+ (k-2)=(n-2)+k.
\end{eqnarray*}
Thus, $k\leq  a_n-(n-1)+1$, and hence by Lemma~\ref{bound} we get $\mu(\ini_<(I(\ab)))\leq {(n-1)+k-1\choose n-2}\leq {a_n\choose n-2}$.
\end{proof}

\section{Applications}
\label{sec:applications}
In this section we use the criteria in Theorem \ref{thm:main} to test the arithmetically Cohen-Macaulay property for two families of projective monomial curves.

\subsection{Bresinsky semigroups}
 In  \cite{Bresinsky} Bresinsky introduced the  semigroup
\begin{equation*}
B_h=\langle  (2h-1)2h,  (2h-1)(2h+1), 2h(2h+1), 2h(2h+1)+ 2h-1 \rangle,
\end{equation*}
where $h\geq 2$. He showed that the toric ideal $I_{B_h} \subset S=K[x,y,z,t]$ is minimally generated by more than $2h$ binomials.
Based on that, in \cite[Section 3.3]{HeS} it is proved  that
$$
 \mathcal{F}=\{xt-yz\} \cup \{z^{i-1}t^{2h-i}-x^{i+1}y^{2h-i}:1\leq i \leq 2h \} \cup \{x^{2h+1-j}z^j-y^{2h-j}t^j: 0\leq j \leq 2h-2\}
$$
is a minimal generating set for $I_{B_h}$.
 Combining the generators  corresponding to $i=2h$ and $j=1$ we get that
$$
  u=-z(z^{2h-1}-x^{2h+1})-x(x^{2h}z-y^{2h-1}t)= xy^{2h-1}t-z^{2h} \in I_{B_h},
$$
hence
$\ini(u)=xy^{2h-1}t \in \ini(I_{B_h})$, where one uses the reverse lexicographic monomial order with $x>y>z>t$.

If  the projective monomial curve associated to $B_h$ were Cohen-Macaulay, then, as the generators of $B_h$ above are listed increasingly, by Theorem \ref{thm:main} we obtain that $xy^{2h-1} \in \ini(I_{B_h})$.
The ideal $I_{B_h}$ is generated by binomials with disjoint support, hence in the reduced Gr\"obner basis of $I_{B_h}$ there exists $v=xy^d-z^\alpha t^\beta$ with $\ini(v)=xy^d$, $0< d\leq 2h-1$ and $\alpha, \beta$ nonnegative integers.
 We denote $a_1,\dots, a_4$ the generators of $B_h$ in the given order.
Since $a_1<a_2<a_3<a_4$ we have $(d+1) a_2> a_1+ d a_2 =\alpha a_3+ \beta a_4 > (\alpha + \beta) a_2$, hence $\alpha+\beta \leq 2h-1$.

If $\alpha+\beta= 2h-1$, after  adding to $v$ the binomial $z^\alpha t^\beta-y^\beta x^{\alpha+2}$ from the given minimal generating set of $I_{B_h}$, we obtain that $xy^d-  x^{\alpha+2} y^\beta \in I_{B_h}$. Thus $\beta \leq d$ and
$y^{d-\beta}-x^{\alpha+1} \in I_{B_h}$, which is false, since $d<2h$ and one can see from $\mathcal{F}$ that  $2h\cdot a_2$ is the smallest positive multiple
of $a_2$ which is in the semigroup generated by $a_1, a_3, a_4$.

Thus $\alpha+\beta < 2h-1$. If we denote $\bar{I}=I_{B_h} \mod  x \subset K[y,z,t]$, then given  $\mathcal{F}$ it follows that $\bar{I}=(yz)+(t,z)^{2h-1}+y^2(y,t)^{2h-1}$.
It is easy to see that the monomial $\bar{v}=z^\alpha t^\beta$ is not in $\bar{I}$, which is a contradiction.

Therefore, we proved the following proposition, which was first obtained by Cavaliere and Niesi in \cite[Remark 5.4]{CN}, as an application of their criterion from Corollary \ref{cor:cn}.

\begin{Proposition} (Cavaliere and Niesi, \cite{CN})
\label{prop:bresisnsky}
The projective monomial curve associated to $B_h$ is not arithmetically Cohen-Macaulay, for any $h \geq 2$.
\end{Proposition}

\subsection{Arslan semigroups}

For $h\geq 2$, let
$$
A_h=\langle h(h+1), h(h+1)+1, (h+1)^2, (h+1)^2+1 \rangle.
$$
This family of numerical semigroups was studied by Arslan  who shows
  in \cite[Proposition 3.2]{Arslan} that the defining ideal of $K[A_h]$ is
\begin{multline}
\label{eq:eq-arslan}
I_{A_h}= (  x^{h-i}z^{i+1}-y^{h-i+1}t^i: 0\leq i < h) + \\
			 (  x^{i+1}y^{h-i} - z^i t^{h-i} : 0\leq i \leq h)+  (  xt-yz).
\end{multline}

\begin{Proposition}
\label{prop:arslan}
The  projective monomial curve associated to $A_h$ is aritmetically Cohen-Macaulay, for any $h\geq 2$.
\end{Proposition}

\begin{proof}

Letting $g_i=   x^{h-i}z^{i+1}-y^{h-i+1}t^i$  for $0\leq i \leq  h$, $f_i= x^{i+1}y^{h-i} - z^i t^{h-i}$ for $0\leq i \leq h$ and $f=xt-yz$, we claim that
$\mathcal{G}=\{ g_0, \dots, g_h, f_0, \dots, f_h, f\}$ is the reduced Gr\"obner basis of $I_{A_h}$ with respect to the reverse lexicographic term order with $x>y>z>t$.
As a consequence, by inspecting the leading monomials we may use Theorem \ref{thm:main} to conclude 
the desired statement.

We   show that all the $S$-pairs of binomials in $\mathcal{G}$ reduce to $0$ with respect to $\mathcal{G}$.
Therefore, by Buchberger's criterion (\cite[Theorem 2.14]{EH})  it follows that  $\mathcal{G}$ is a Gr\"obner basis for $I_{A_h}$.

$S(g_0, f)= zg_0+ y^h f=x^hz^2-xy^ht=xg_1  \stackrel{\mathcal{G}}{\rightarrow} 0$. \\
For $1 \leq i \leq h$:

 $S(g_i, f)=yg_i- x^{h-i}z^if= x^{h-i+1}z^it- y^{h-i+2}t^i= tg_{i-1}  \stackrel{\mathcal{G}}{\rightarrow} 0$. \\
For $0\leq i <h$:

$S(f_i, f)= z f_i-x^{i+1}y^{h-i-1}f=x^{i+2}y^{h-i-1}t-z^{i+1}t^{h-i}=t f_{i+1}  \stackrel{\mathcal{G}}{\rightarrow} 0$.\\
Also,
$S(f_h, f)  \stackrel{\mathcal{G}}{\rightarrow} 0$ since $\gcd(\ini(f_h), \ini(f))=\gcd (x^{h+1}, yz)=1$.\\
For $1\leq i \leq h$:

$S(g_i, g_0)  \stackrel{\mathcal{G}}{\rightarrow} 0$ since $\gcd(\ini(g_i), \ini(g_0))=1$.\\
For $0\leq i\leq  h$:

$S(f_i, g_0)= y^{i+1}f_i+x^{i+1}g_0= x^{h+i+1}z-y^{i+1}z^i t^{h-i}=  x^i(x^{h+1}-z^h)z+   z^i(x^i z^{h-i+1}-y^{i+1}t^{h-i})=x^i z f_a+ z^i g_{h-i}
 \stackrel{\mathcal{G}}{\rightarrow} 0.$\\
For $0\leq j<i \leq h$:

$S(g_j, g_i)= z^{i-j}g_j-x^{i-j}g_i=y^{h-j+1}z^{i-j}t^j-x^{i-j}y^{h-i+1}t^i= y^{h-i+1}t^j (y^{i-j}z^{i-j}-x^{i-j}t^{i-j})= y^{h-i+1}t^j \cdot f\cdot (\dots)  \stackrel{\mathcal{G}}{\rightarrow} 0$. \\
For $1\leq i \leq a$, $0\leq j \leq a$ with $i\leq h-j-1$, i.e. $i+j <h$:

$ S(f_j, g_i)= x^{h-i-j-1}z^{i+1}f_j - y^{h-j} g_i =y^{2h-i-j+1}t^{i}-x^{h-i-j-1}z^{i+j+1}t^{h-j}=y^{h-i-j} t^{i} (y^{h+1}-x^hz) +x^{h-i-j-1}z t^i (x^{i+j+1}y^{h-i-j}-z^{j+i}t^{h-j-i})
=-t^i y^{h-i-j}g_0+ x^{h-i-j-1}zt^i f_{i+j}  \stackrel{\mathcal{G}}{\rightarrow} 0$.\\
For $1\leq i \leq a$, $0\leq j \leq a$ with $i> h-j-1$, i.e. $i+j \geq h$:

$S(f_j, g_i)= z^{i+1}f_j-x^{i+j+1-h}y^{h-j}g_i= x^{i+j+1-h}y^{2h-i-j+1}t^i-z^{i+j+1}t^{h-j} = \\ yt^i  (x^{i+j+1-h}y^{2h-i-j}- z^{i+j-h}t^{2h-i-j}) + z^{i+j-h}t^{h-j} (yt^h-z^{h+1})  = yt^i f_{i+j-h} + z^{i+j-h}t^{h-j} g_h  \stackrel{\mathcal{G}}{\rightarrow} 0$.\\
For $0\leq j< i \leq a$:

$S(f_j, f_i)= x^{i-j}f_j-y^{i-j}f_i= y^{i-j}z^i t^{h-i}-x^{i-j}z^j t^{h-j}= z^j t^{h-i} (y^{i-j}z^{i-j}-x^{i-j}t^{i-j})= z^j t^{h-i} (yz-xt) \cdot (\dots)=  z^j t^{h-i}\cdot f \cdot (\dots)   \stackrel{\mathcal{G}}{\rightarrow} 0$.
\end{proof}

\medskip
{\bf Acknowledgement}.
We gratefully acknowledge the use of the  Singular (\cite{Sing}) software for our computations.
 Dumitru Stamate  was supported by the University of Bucharest, Faculty of Mathematics and
Computer Science, through the 2017 Mobility Fund.
{}
\end{document}